\documentclass[12pt, reqno]{amsart}
\usepackage{amsmath, amsthm, amscd, amsfonts, amssymb, graphicx, color}
\usepackage[bookmarksnumbered, colorlinks, plainpages]{hyperref}
\hypersetup{colorlinks=true,linkcolor=red, anchorcolor=green, citecolor=cyan, urlcolor=red, filecolor=magenta, pdftoolbar=true}

\usepackage{tensor}

\newtheorem{theorem}{Theorem}[section]

\newtheorem{corollary}[theorem]{Corollary}
\theoremstyle{definition}
\newtheorem{definition}[theorem]{Definition}
\newtheorem{example}[theorem]{Example}

\theoremstyle{remark}
\newtheorem{remark}[theorem]{Remark}
\numberwithin{equation}{section}

\usepackage{fullpage}

\begin{document}
	
	\setcounter{page}{1}
	
	\title[FPT on metric type spaces]{Weak contractions via $\lambda$-sequences.}

	\author[Collins Amburo Agyingi]{Collins Amburo Agyingi$^{1,3}$}
	
	\author[Ya\'e Ulrich Gaba]{Ya\'e Ulrich Gaba$^{1,2,3,\dagger}$}

	\address{$^{1}$ Department of Mathematical Sciences, North West University, Private Bag
		X2046, Mmabatho 2735, South Africa.}

	\address{$^{2}$ Institut de Math\'ematiques et de Sciences Physiques (IMSP), 01 BP 613 Porto-Novo, B\'enin.}

	\address{$^{3}$ African Center for Advanced Studies (ACAS),
		P.O. Box 4477, Yaounde, Cameroon.}

	\email{\textcolor[rgb]{0.00,0.00,0.84}{collins.agyingi@nwu.ac.za
	}}
	
	\email{\textcolor[rgb]{0.00,0.00,0.84}{yaeulrich.gaba@gmail.com
	}}

	\subjclass[2010]{Primary 47H05; Secondary 47H09, 47H10.}
	
	\keywords{metric type space, fixed point, $\lambda$-sequence.}

	\subjclass[2010]{Primary 47H10; Secondary 54H25.}
	
	\keywords{metric type space, common fixed point, $\lambda$-sequence.}
	
	\date{Received: xxxxxx; Accepted: zzzzzz.
		\newline \indent $^{\dagger}$Corresponding author}
	
	\begin{abstract}
		In this note, we discuss common fixed point for a family of self mapping defined on a metric type space and satisfying a weakly contractive condition. In our development, we make use of the $\lambda$-sequence approach and also of a certain class of real valued maps. We derive some implications for self-mappings on  quasi-pseudometric type spaces.
	\end{abstract} 
	
	\maketitle
	
	\section{Introduction and preliminaries}
	
Weakly $C$-contractive maps have been introduced by Choudhury\cite{bina1} as a modified version of the contactive conditions proposed by Chatterjee\cite{chatte} and Kannan\cite{kan1,kan2}. All these different types of contractive conditions have been showed to be independent of one another (see \cite{rod}) and most of these results have been discussed in considering a single self-map on a metric space. We intend, in the present, to extend the result of Choudhury\cite{bina1} by considering a family of self-maps on a metric type space. The paper by Jovanovi\'c et al.\cite{jova} is an interesting reference regarding common fixed point in metric type spaces.

The following is the definition in a metric space of a weakly $C$-contractive map as it appears in Choudhury\cite[Definition 1.3]{bina1}.

\begin{definition}
	A mapping $T : (X, m) \to (X, m)$, where $(X, m)$ is a metric space is said to be weakly $C$-contractive
	or is a weak $C$-contraction if for all $x, y \in X$,
	the following inequality holds:
	
	\begin{equation*}
	m(Tx,Ty) \leq \frac{1}{2}[m(x,Ty)+m(y,Tx)]- \psi(m(x,Tx),m(y,Ty)),
	\end{equation*}
	where $\psi:[0,\infty)^2 \to [0,\infty)$ is a continuous mapping such that $%
	\psi(x,y)=0$ if and only if $x=y=0$.
\end{definition}	
	
It is important to point out that weak $C$-contractions constitute a very large class of mappings that contains the above mentioned ones.
		
Next, we recall the definition of a metric type space.

\begin{definition}(See \cite[Definition 1.1]{gab})
	Let $X$ be a nonempty set, and let the function $D:X\times X \to [0,\infty)$ satisfy the following properties:
	\begin{itemize}
		\item[(D1)] $D(x,x)=0$ for any $x \in X$;
		\item[(D2)] $D(x,y)=D(y,x)$ for any $x,y\in X$;
		\item[(D3)] $D(x,y) \leq K \big( D(x,z_1)+D(z_1,z_2)+\cdots+D(z_n,y) \big)$ for any points $x,y,z_i\in X,\ i=1,2,\ldots, n$ where $n\geq 1$ is a fixed natural number and some constant $K\geq 1$.
	\end{itemize}
	The triplet $(X,D,K)$ is called a \textbf{metric type space}.
\end{definition}
The class of metric type spaces is strictly larger than that of metric spaces (see \cite{niyi-gaba}). The concepts of Cauchy sequence, convergence for a sequence and completeness in a metric type space are defined in the same way as defined for a metric space.

We also recall the definition of a $\lambda$-sequence.

\begin{definition}\label{def1}(See \cite[Definition 2.1]{gab})
	A sequence $(x_n)_{n\geq 1}$ in a metric type space $(X,D,K)$ is a $\lambda$-sequence if there exist $0<\lambda<1$ and $n(\lambda) \in \mathbb{N}$ such that $$\sum_{i=1}^{L} D(x_i,x_{i+1}) \leq \lambda L \text{ for each } L\geq n(\lambda).$$
\end{definition}

Finally, we denote by $\Phi$ be the class of continuous, non-decreasing, sub-additive and homogeneous functions $F:[0,\infty) \to [0,\infty)$ such that $F^{-1}(0)=\{0\}$.

	We can now state our fixed point results.	
	
	\section{Main results}


\vspace*{0.3cm}

\subsection{Common fixed point theorems ( Kannan-Choudhury case)}
	
\hspace*{2cm}	

\vspace*{0.2cm}
	
	In this section, we prove existence of a unique common
	fixed point for a family of contractive type self-maps on a
	complete metric type space by using the Kannan contractive condition as a base.

\begin{theorem}\label{main}

Let $(X,D,K)$ be a complete metric type space and $\{T_n\}$ be a sequence of self mappings on $X$. Assume that there exist two sequences $(a_n)$ and $(b_n)$ of elements of $X$ such that

\begin{align}\label{conditionqq}
F(D(T_i(x),T_j(y))) \leq \  & F(\delta_{i,j}[D(x,T_i(x))+ D(y,T_j(y))])\\
&- F(\gamma_{i,j}\psi[D(x,T_i(x)),D(y,T_j(y))]) \nonumber
\end{align}

for $x,y\in X$ with $x\neq y,$ $0\leq \delta_{i,j},\gamma_{i,j}<1 , \ i,j = 1,2,\cdots ,$ and for some $F \in \Phi$ homogeneous with degree $s$, where $\delta_{i,j}=D(a_i,a_j)$, $\gamma_{i,j}=D(b_i,b_j)$, and $\psi:[0,\infty)^2 \to [0,\infty)$ is a continuous mapping such that $%
\psi(x,y)=0$ if and only if $x=y=0$. If the sequence $(s_n)$ where $s_i=\frac{\delta_{i,i+1}^s }{1-\delta_{i,i+1}^s}$ is a non-increasing $\lambda$-sequence of elements of $\mathbb{R}^+=[0,+\infty)$ endowed with the $\max\footnote{The max metric $m$ refers to $m(x,x)=0$ and $m(x,y)=\max\{x,y\}$ if $x\neq y$ for any $x,y \in [0,\infty)$}$ metric,
then $\{T_n\}$ has a unique common fixed point in $X$. 

\end{theorem}

\begin{proof}
	
	For any $x_0\in X$, we construct the Picard type sequence $(x_n)$ by setting $x_n= T_n(x_{n-1}),\ n=1,2,\cdots .$ Using \eqref{conditionqq} and the homogeneity of $F$, we obtain
	
	\begin{align*}
	F(D(x_1,x_2))  = \ & F(D(T_1(x_0),T_2(x_1))) \\
	\leq \ &\delta_{1,2}^s F([D(x_0,T_1(x_0)))+ D(x_1,T_2(x_1))]) \\
	& - \gamma_{1,2}^sF(\psi[D(x_0,T_1(x_0))), D(x_1,T_2(x_1))]) \\
	= \ & \delta_{1,2}^sF([D(x_0,x_1)+ D(x_1,x_2)] )-  \gamma_{1,2}^sF(\psi[D(x_0,x_1), D(x_1,x_2)])\\
	 \leq & \delta_{1,2}^sF([D(x_0,x_1)+ D(x_1,x_2)] )
	\end{align*}
	
	Therefore, using the sub-additivity of $F$, we deduce that
	\[ (1-\delta_{1,2}^s)F(D(x_1,x_2)) \leq (\delta_{1,2}^s)   F(D(x_0,x_1)),\] 
	i.e.
	
	\[  F(D(x_1,x_2)) \leq \left(\frac{\delta_{1,2}^s}{1-\delta_{1,2}^s}   \right) F(D(x_0,x_1)) .\]
	
	Also, we get
	
	\begin{align*}
	F(D(x_2,x_3)) & = F(D(T_2(x_1),T_3(x_2))) \\
	& \leq \left(\frac{\delta_{2,3}^s}{1-\delta_{2,3}^s}   \right) F(D(x_1,x_2)) \\
	& \leq  \left(\frac{\delta_{2,3}^s}{1-\delta_{2,3}^s}   \right) \left(\frac{\delta_{1,2}^s}{1-\delta_{1,2}^s}   \right) F(D(x_0,x_1)).
	\end{align*}
	
	By repeating the above process, we have
	
	\begin{equation}
	F(D(x_n,x_{n+1})) \leq \prod\limits_{i=1}^n \left(\frac{\delta_{i,i+1}^s}{1-\delta_{i,i+1}^s}   \right)  F(D(x_0,x_1)).
	\end{equation}

Hence we derive, by making use of the the triangle inequality and the properties of $F$, that for $p>0$

\begin{align*}
F(D(x_n,x_{n+p})) \leq \ & K^s[F(D(x_n,x_{n+1})) +F(D(x_{n+1},x_{n+2})) \\
& + \ldots + F(D(x_{n+p-1},x_{n+p}))]  \\
\leq \ & K^s\left[\prod\limits_{i=1}^n \left(\frac{\delta_{i,i+1}^s}{1-\delta_{i,i+1}^s}   \right)  F(D(x_0,x_1)) \right. \\
& + \prod\limits_{i=1}^{n+1} \left(\frac{\delta_{i,i+1}^s}{1-\delta_{i,i+1}^s}   \right)  F(D(x_0,x_1))        \\
& + \ldots + \\
& + \left.\prod\limits_{i=1}^{n+p-1} \left(\frac{\delta_{i,i+1}^s}{1-\delta_{i,i+1}^s}   \right)  F(D(x_0,x_1))\right] \\
= \ & K^s\left[ \sum\limits_{k=0}^{p-1} \prod\limits_{i=1}^{n+k} \left(\frac{\delta_{i,i+1}^s}{1-\delta_{i,i+1}^s}   \right)  F(D(x_0,x_1))\right] \\
= \ & K^s\left[\sum\limits_{k=n}^{n+p-1} \prod\limits_{i=1}^{k} \left(\frac{\delta_{i,i+1}^s}{1-\delta_{i,i+1}^s}   \right)  F(D(x_0,x_1))\right].
\end{align*}

Now, let $\lambda$ and $n(\lambda)$ as in Definition \ref{def1}, then for $n\geq n(\lambda)$ and using the fact that the geometric mean of non-negative real numbers is at most their arithmetic mean, it follows that 

\begin{align}
F(D(x_n,x_{n+p})) \leq \ & K^s\left[\sum\limits_{k=n}^{n+p-1} \left[ \frac{1}{k}\sum\limits_{i=1}^{k} \left(\frac{\delta_{i,i+1}^s}{1-\delta_{i,i+1}^s}   \right) \right]^k F(D(x_0,x_1))\right] \\
\leq \ & K^s\left[ \left(\sum\limits_{k=n}^{n+p-1} \lambda^k \right) F(D(x_0,x_1))\right] \nonumber \\
\leq \ & K^s \frac{\lambda^n}{1-\lambda}F(D(x_0,x_1)) \nonumber .
\end{align}
Letting $n\to \infty$ and since $F^{-1}(0)=\{0\}$ and $F$ is continuous, we deduce that $D(x_n,x_{n+p}) \to 0.$ Thus $(x_n)$ is a Cauchy sequence and, by completeness of $X$, converges to say $x^* \in X$.

Moreover, for any natural number $m\neq 0$, we have

\begin{align*}
F(D(x_n, T_m(x^*))) = \ & F(D(T_n(x_{n-1}),T_m(x^*))) \\
\leq \ & \delta_{n,m}^s[F(D(x_{n-1},x_n))+ F(D(x^*,T_m(x^*)))] \\
& - \gamma_{n,m}^sF(\psi[D(x_{n-1},x_n),D(x^*,T_m(x^*)]).
\end{align*}

Again, letting $n \to \infty$, we get

\begin{align*}
F(D(x^*, T_m(x^*))) & \leq \  \delta_{n,m}^s[F(D(x^*,x^*))+ F(D(x^*,T_m(x^*)))] - \gamma_{n,m}^sF(\psi[0,D(x^*,T_m(x^*)]) \\
& \leq \ \delta_{n,m} ^sF(D(x^*,T_m(x^*)),
\end{align*}
and since $0\leq \delta_{n,m} <1$, it follows that $F(D(x^*, T_m(x^*)))=0$, i.e. $T_m(x^*)=x^*$.

Then $x^*$ is a common fixed point of $\{T_m\}_{m\geq 1}$. 

To prove the uniqueness of $x^*$, let us suppose that $y^*$ is a common fixed point of $\{T_m\}_{m\geq 1}$, that is $T_m(y^*)=y^*$ for any $m\geq 1$. Then, by \eqref{conditionqq}, we have

\begin{align*}
F(D(x^*,y^*))  & \leq \ F(D(T_m(x^*),T_m(y^*))) \\
& \leq \ \delta_{n,m}^s [F(D(x^*,T_m(x^*)) + F(D(y^*,T_m(y^*))] - \gamma_{n,m}^sF(\psi[D(x^*,T_m(x^*), D(y^*,T_m(y^*))]) \\
& = \delta_{n,m}^s [F(0) + F(0)] - \gamma_{n,m}^sF(\psi[0,0])\\
& = 0.  
\end{align*}

So $x^*$ is the unique common fixed point of $\{T_m\}$.

\end{proof}

	As particular cases of Theorem \ref{main}, we have the following two corollaries.

	\begin{corollary}\label{cor1}
		Let $(X,D,K)$ be a complete metric type space and $\{T_n\}$ be a sequence of self mappings on $X$. Assume that there exist two sequences $(a_n)$ and $(b_n)$ of elements of $X$ such 
		
		\begin{align}\label{condition}
		D(T_i(x),T_j(y)))\leq \  & \delta_{i,j}[D(x,T_i(x))+ D(y,T_j(y))]\\
		&- \gamma_{i,j}\psi[D(x,T_i(x)),D(y,T_j(y))] \nonumber
		\end{align}
		
		for $x,y\in X$ with $x\neq y,$ $0\leq \delta_{i,j},\gamma_{i,j}<1 , \ i,j = 1,2,\cdots ,$ where $\delta_{i,j}=D(a_i,a_j)$, $\gamma_{i,j}=D(b_i,b_j)$, and $\psi:[0,\infty)^2 \to [0,\infty)$ is a continuous mapping such that $%
		\psi(x,y)=0$ if and only if $x=y=0$. 	If the sequence $(s_n)$ where $s_i=\frac{\delta_{i,i+1}^s}{1-\delta_{i,i+1}^s}$ is a non-increasing $\lambda$-sequence of elements of $\mathbb{R}^+=[0,+\infty)$ endowed with the $\max$ metric, 
		then $\{T_n\}$ has a unique common fixed point in $X$. 
		
	\end{corollary}

\begin{proof} 
Apply Theorem \ref{main} by putting $F=I_{[0,\infty)}$, the identity map.	
\end{proof}

	\begin{corollary}\label{cor2}
		Let $(X,D,K)$ be a complete metric type space and $\{T_n\}$ be a sequence of self mappings on $X$. Assume that there exists a sequences $(a_n)$ of elements of $X$ such 
		
		\begin{align}\label{condition2wwwwww}
		F(D(T_i(x),T_j(y))) \leq \  & F(\delta_{i,j}[D(x,T_i(x))+ D(y,T_j(y))])
		\end{align}
		
		for $x,y\in X$ with $x\neq y,$ $0\leq \delta_{i,j}<1 , \ i,j = 1,2,\cdots ,$ and for some $F \in \Phi$ homogeneous with degree $s$, where $\delta_{i,j}=D(a_i,a_j)$. If the sequence $(s_n)$ where $s_i=\frac{\delta_{i,i+1}^s}{1-\delta_{i,i+1}^s}$ is a non-increasing $\lambda$-sequence of elements of $[0,+\infty)$ endowed with the $\max$ metric,
	
		then $\{T_n\}$ has a unique common fixed point in $X$. 
		
	\end{corollary}
	
	\begin{proof} 
	Apply Theorem \ref{main} by putting\footnote{In this case, we can choose $(b_n)$ to be any constant sequence of elements of $X$.} $\gamma_{i,j}=0$.
\end{proof}

	A more general\footnote{Natural in some sense} weak contraction could also be considered as we write our next result.

	\begin{theorem}\label{main2}
		Let $(X,D,K)$ be a complete metric type space and $\{T_n\}$ be a sequence of self mappings on $X$. Assume that there exist two sequences $(a_n)$ and $(b_n)$ of elements of $X$ such 
		
		\begin{align}\label{condition2}
		F(D(T_i(x),T_j(y))) \leq \  & F(\delta_{i,j}[D(x,T_i(x))+ D(y,T_j(y))+D(x,y)])\\
		&- F(\gamma_{i,j}\psi[D(x,T_i(x)),D(y,T_j(y)),D(x,y)]) \nonumber
		\end{align}
		
		for $x,y\in X$ with $x\neq y,$ $0\leq \delta_{i,j},\gamma_{i,j}<1 , \ i,j = 1,2,\cdots ,$ and for some $F \in \Phi$ homogeneous with degree $s$, where $\delta_{i,j}=D(a_i,a_j)$, $\gamma_{i,j}=D(b_i,b_j)$, and $\psi:[0,\infty)^3 \to [0,\infty)$ is a continuous mapping such that $%
		\psi(x,y,z)=0$ if and only if $x=y=z=0$. If the sequence $(s_n)$ where $s_i=\frac{2^s\delta_{i,i+1}^s }{1-\delta_{i,i+1}^s}$ is a non-increasing $\lambda$-sequence of elements of $\mathbb{R}^+=[0,+\infty)$ endowed with the $\max$ metric,
		then $\{T_n\}$ has a unique common fixed point in $X$. 
		
\end{theorem}

We shall omit the proof as it is merely a copy of that of Theorem \ref{main}. Then we also give, without proof the following corollary.

\begin{corollary}\label{cor11}
	Let $(X,D,K)$ be a complete metric type space and $\{T_n\}$ be a sequence of self mappings on $X$. Assume that there exists a sequence $(a_n)$ of elements of $X$ such 
	
	\begin{align}\label{condition22}
	F(D(T_i(x),T_j(y)))\leq \  &F( \delta_{i,j}[D(x,T_i(x))+ D(y,T_j(y))+D(x,y)]),\\
	\end{align}
	
	for $x,y\in X$ with $x\neq y,$ $0\leq \delta_{i,j}<1 , \ i,j = 1,2,\cdots ,$ where $\delta_{i,j}=D(a_i,a_j)$, and $\psi:[0,\infty)^3 \to [0,\infty)$ is a continuous mapping such that $%
	\psi(x,y,z)=0$ if and only if $x=y=z=0$. If the sequence $(s_n)$ where $s_i=\frac{2^s\delta_{i,i+1}^s}{1-\delta_{i,i+1}^s}$ is a non-increasing $\lambda$-sequence of elements of $\mathbb{R}^+=[0,+\infty)$ endowed with the $\max$ metric,
	then $\{T_n\}$ has a unique common fixed point in $X$. 
	
\end{corollary}

	\begin{example}
		Let $X=[0,1]$ and $D$ be the metric defined by $D(x,x)=0$ and $D(x,y)= \max\{x,y\}$ if $x\neq y$. Clearly, $(X,D)$ is a complete metric space. Just observe that any Cauchy sequence in $(X,D)$ converges to $0$. Following the notation in Theorem \ref{main}, we set $a_i = \left(\frac{1}{1+2^i}\right)^2$ so that $\delta_{i,j}=\left(\frac{1}{1+2^\eta}\right)^2$ where $\eta= \min\{i,j\}$. We also define $T_i(x)= \frac{x}{16^i}$ for all $x\in X$ and $i=1,2,\cdots$ and $F:[0,\infty) \to [0,\infty), \ x\mapsto \sqrt{x}$. Then $F$ is continuous, non-decreasing, sub-additive and homogeneous with degree $s=\frac{1}{2}$ and $F^{-1}(0)=\{0\}$.

		Assume $i<j$ and $x>y$. Hence we have
		
		\[ F(D(T_i(x),T_j(y)))  = \sqrt{\frac{x}{16^i}} \]
		and
		
		\[ F( \delta_{i,j} [D(x,T_i(x))+ D(y,T_j(y))+D(x,y)]  ) =  \sqrt { \left(\frac{1}{1+2^i}\right)^2 (2x+y)} .\]
		
		Therefore condition \eqref{condition22} is satisfied for all $x,y\in X$ with $x\neq y$. Moreover, since $F$ is homogeneous with degree $s=\frac{1}{2}$, the sequence
		$$ s_i=\frac{ 2^s \delta_{i,i+1}^s }{1-\delta_{i,i+1}^s} = \frac{\sqrt{2}}{2^i} $$
		is a $\lambda$-sequence with $\lambda=\frac{\sqrt{2}}{2}$. Then by Corollary \ref{cor11}, $\{T_n\}$ has a common fixed point, which is this case $x^*=0.$
	\end{example}
	
	\vspace*{0.5cm}

The contractive condition in Theorem \ref{main} can be relaxed and we obtain the next result: 

\begin{theorem}\label{relaxed}
	Let $(X,D,K)$ be a complete metric type space and $\{T_n\}$ be a sequence of self mappings on $X$ such that
	
	\begin{align}\label{conditionrelaxed}
	F(D(T_i(x),T_j(y))) \leq \  & F(\delta_{i,j}[D(x,T_i(x))+ D(y,T_j(y))])\\
	&- F(\gamma_{i,j}\psi[D(x,T_i(x)),D(y,T_j(y))]) \nonumber
	\end{align}

for $x,y\in X$ with $x\neq y,$ $0\leq \delta_{i,j}, \gamma_{i,j} , \ i,j = 1,2,\cdots ,$ for some $F \in \Phi$ homogeneous with degree $s$ and $\psi:[0,\infty)^2 \to [0,\infty)$ is a continuous mapping such that $%
\psi(x,y)=0$ if and only if $x=y=0$.
If 
\begin{itemize}
	\item[i)] for each $j,$ $\underset{i \to \infty}{\limsup}\ \delta_{i,j}^s <1, $
	
	\item[ii)] $$\sum_{n=1}^{\infty} C_n < \infty \text{ where } C_n = \prod\limits_{i=1}^n \frac{\delta_{i,i+1}^s}{1-\delta_{i,i+1}^s}, $$
\end{itemize}

then $\{T_n\}$ has a unique common fixed point in $X$. 

\end{theorem}

	\begin{proof}
	Just observe that the Picard type sequence $(x_n)=(T_n(x_{n-1}))$ for any initial $x_0\in X$ is such that 
	
	\begin{equation}
	F(D(x_n,x_{n+1})) \leq \prod\limits_{i=1}^n \left(\frac{\delta_{i,i+1}^s}{1-\delta_{i,i+1}^s}   \right)  F(D(x_0,x_1)) =: C_n F(D(x_0,x_1)),
	\end{equation}
and so for $p>1$

\begin{align*}
F(D(x_n,x_{n+p})) \leq K^{s}\left[\sum\limits_{k=n}^{n+p-1} C_k\right] F(D(x_0,x_1)).
\end{align*}
Letting $n\to \infty$, we deduce that $D(x_n,x_{n+p}) \to 0.$ Thus $(x_n)$ is a Cauchy sequence and, by completeness of $X$, converges to say $x^* \in X$. It is easy to see that $x^*$ is the unique common fixed point of $\{T_m\}$.	
	
\end{proof}

\begin{corollary}\label{correlax}
	Let $(X,D,K)$ be a complete metric type space and $\{T_n\}$ be a sequence of self mappings on $X$ such that
	
	\begin{align}\label{condcorrelaxed}
	F(D(T_i(x),T_j(y))) \leq \  & F(\delta_{i,j}[D(x,T_i(x))+ D(y,T_j(y))])
	\end{align}
	
	for $x,y\in X$ with $x\neq y,$ $0\leq \delta_{i,j}<1 , \ i,j = 1,2,\cdots ,$ for some $F \in \Phi$ homogeneous with degree $s$ and $\psi:[0,\infty)^2 \to [0,\infty)$ is a continuous mapping such that $%
	\psi(x,y)=0$ if and only if $x=y=0$.
	
	If 
	\begin{itemize}
		\item[i)] for each $j,$ $\underset{i \to \infty}{\limsup}\ \delta_{i,j}^s <1, $
		
		\item[ii)] $$\sum_{n=1}^{\infty} C_n < \infty \text{ where } C_n = \prod\limits_{i=1}^n \frac{\delta_{i,i+1}^s}{1-\delta_{i,i+1}^s}, $$
	\end{itemize}
	
	then $\{T_n\}$ has a unique common fixed point in $X$. 	
\end{corollary}		

\begin{proof}
Apply Theorem \ref{relaxed} by putting $\gamma_{i,j}=0$.
\end{proof}

\begin{example}

	Let $X=[0,1]$ and $D(x,y)= |x-y|$ whenever $x,y \in [0,1]$. Clearly, $(X,G)$ is a complete metric space. 

	Following the notation in Theorem \ref{relaxed}, we set $\delta_{i,j} = \left(\frac{1}{1+2^i}\right)^2$.
	We also define $T_i(x)= \frac{x}{4^i}$ for all $x\in X$ and $i=1,2,\cdots$ and $F:[0,\infty) \to [0,\infty), \ x\mapsto \sqrt{x}$. Then $F$ is continuous, non-decreasing, sub-additive and homogeneous of degree $s=\frac{1}{2}$ and $F^{-1}(0)=\{0\}$. Assume $i<j$ and $x>y\geq z$. Hence we have
	\[
	F(D(T_i(x),T_j(y))  = \sqrt{\left|\frac{x}{4^i}- \frac{y}{4^j}     \right|}
	\]
	
	and
	
	\[ F( \delta_{i,j} [D(x,T_i(x))+ D(y,T_j(y))]  ) =  \sqrt { \left(\frac{1}{1+2^i}\right)^2 \left(\left|x-\frac{x}{4^i}\right|+ \left| y-\frac{y}{4^j} \right| \right)} .\]
	
	Therefore condition \eqref{condcorrelaxed} is satisfied for all $x,y\in X$ with $x\neq y$. Moreover, since $F$ is homogeneous of degree $s=\frac{1}{2}$, we have 
	\begin{itemize}
		\item[i)] $\underset{i \to \infty}{\limsup}\ \delta^s_{i,j} <1, $

		\item[ii)] $$ C_n= \prod\limits_{i=1}^n \frac{\delta_{i,i+1}^s}{1-\delta_{i,i+1}^s} =  \prod\limits_{i=1}^n \frac{1}{2^i}=  \frac{1}{2^{\frac{n(n+1)}{2}}} \leq \frac{1}{2 ^n}.$$
		
	\end{itemize}
	
	The conditions of Corollary \ref{correlax} are satisfied, $\{T_n\}$ has a common fixed point, which is this case $x^*=0.$

\end{example}

An interesting direction to look into is that of the quasi-pseudometric type spaces which were investigated by 	Kazeem et al. \cite{kaz}. In \cite{kaz}, we can read the following definition for a quasi-pseudometric type space:

	\begin{definition}(Compare \cite[Definition 29]{kaz})
		Let $X$ be a nonempty set, and let the function $D:X\times X \to [0,\infty)$ satisfy the following properties:
		\begin{itemize}
			\item[(q1)] $D(x,x)=0$ for any $x \in X$;
			\item[(q2)] $D(x,y) \leq K \big( D(x,z_1)+D(z_1,z_2)+\cdots+D(z_n,y) \big)$ for any points $x,y,z_i\in X,\ i=1,2,\ldots, n$ where $n\geq 1$ is a fixed natural number and some constant $K\geq 0$.
		\end{itemize}
		The triplet $(X,D,K)$ is called a \textbf{quasi-pseudometric type space}. Moreover, if $$D(x,y)=0=D(y,x) \Longrightarrow x=y,$$ then $D$ is said to be a $T_0$-quasi-metric type. 
	\end{definition}
	
\begin{remark}
For a given $T_0$-quasi-metric type $D$, it is easy to verify that the function $D^s(x,y) = \max\{D(x,y),D(y,x)\}$ is metric type. Moreover, a $T_0$-quasi-metric type space $(X,D)$ will be called \textbf{bicomplete} if the metric type space $(X,D^s)$ is complete.
\end{remark}

Theorem \ref{main} and Theorem \ref{main2} can be reformulated in the asymmetric setting respectively as:

\begin{theorem}(Compare Theorem \ref{main})
	Let $(X,D,K)$ be a bicomplete quasi-metric type space and $\{T_n\}$ be a sequence of self mappings on $X$. Assume that there exist two sequences $(a_n)$ and $(b_n)$ of elements of $X$ such 
	
	\begin{align}\label{conditionasym}
	F(D(T_i(x),T_j(y))) \leq \  & F(\delta_{i,j}[D(x,T_i(x))+ D(T_j(y),y)])\\
	&- F(\gamma_{i,j}\psi[D^s(x,T_i(x)),D^s(y,T_j(y))]) \nonumber
	\end{align}
	
	for $x,y\in X$ with $x\neq y,$ $0\leq \delta_{i,j},\gamma_{i,j}<1 , \ i,j = 1,2,\cdots ,$ and for some $F \in \Phi$ homogeneous with degree $t$, where $\delta_{i,j}=D^s(a_i,a_j)$, $\gamma_{i,j}=D^s(b_i,b_j)$, and $\psi:[0,\infty)^2 \to [0,\infty)$ is a symmetric continuous mapping such that $%
	\psi(x,y)=0$ if and only if $x=y=0$. If the sequence $(s_n)$ where $s_i=\frac{\delta_{i,i+1}^t }{1-\delta_{i,i+1}^t}$ is a non-increasing $\lambda$-sequence of elements of $\mathbb{R}^+=[0,+\infty)$ endowed with the $\max$ metric,
	then $\{T_n\}$ has a unique common fixed point in $X$. 
	
\end{theorem}

\begin{proof}
	Just observe that for $x,y\in X$ with $x\neq y,$  condition \eqref{conditionasym} gives
	
	\begin{align*}
	F(D^{-1}(T_i(x),T_j(y))) & =  F(D(T_j(y)),T_i(x)) \\
	  & \leq F(\delta_{j,i}[D(y,T_j(y))+D(T_i(x),x) ])\\
	&- F(\gamma_{j,i}\psi[D^s(y,T_j(y)),D^s(x,T_i(x))]) \nonumber \\
	  & = F(\delta_{i,j}[D^{-1}(x,T_i(x))+ D^{-1}(T_j(y),y) ])\\
	&- F(\gamma_{i,j}\psi[D^s(x,T_i(x)),D^s(y,T_j(y))]) \nonumber, \\
	\end{align*}

which implies that

\[	F(D^{s}(T_i(x),T_j(y))) \leq \  F(\delta_{i,j}[D^s(T_i(x),x)+ D(y,T_j(y)) ])- F(\gamma_{i,j}\psi[D^s(x,T_i(x)),D^s(y,T_j(y))]),   \]	
	i.e. the conditions of Theorem \ref{main} are fulfilled. This completes the proof.
\end{proof}

Similarly, we have

\begin{theorem}(Compare Theorem \ref{main2})
		Let $(X,D,K)$ be a bicomplete quasi-metric type space and $\{T_n\}$ be a sequence of self mappings on $X$. Assume that there exist two sequences $(a_n)$ and $(b_n)$ of elements of $X$ such 
	
	\begin{align}
	F(D(T_i(x),T_j(y))) \leq \  & F(\delta_{i,j}[D(x,T_i(x))+ D(T_j(y),y)+D(x,y)])\\
	&- F(\gamma_{i,j}\psi[D^s(x,T_i(x)),D^s(y,T_j(y)),D^s(x,y)]) \nonumber
	\end{align}
	
	for $x,y\in X$ with $x\neq y,$ $0\leq \delta_{i,j},\gamma_{i,j}<1 , \ i,j = 1,2,\cdots ,$ and for some $F \in \Phi$ homogeneous with degree $t$, where $\delta_{i,j}=D^s(a_i,a_j)$, $\gamma_{i,j}=D^s(b_i,b_j)$, and $\psi:[0,\infty)^3 \to [0,\infty)$ is a symmetric continuous mapping such that $%
	\psi(x,y,z)=0$ if and only if $x=y=z=0$. If the sequence $(s_n)$ where $s_i=\frac{2^t\delta_{i,i+1}^t }{1-\delta_{i,i+1}^t}$ is a non-increasing $\lambda$-sequence of elements of $\mathbb{R}^+=[0,+\infty)$ endowed with the $\max$ metric,
	then $\{T_n\}$ has a unique common fixed point in $X$. 
	
\end{theorem}

\vspace*{0.5cm}

\subsection{Common fixed point theorems ( Chatterjea-Choudhury case)} \hspace*{2cm}	

\vspace*{0.2cm}

In this section, we prove existence of a unique common
fixed point for a family of contractive type self-maps on a
complete metric type space by using the Chatterjea contractive condition as a base.

\begin{theorem}\label{chat1}

Let $(X,D,K)$ be a complete metric type space and $\{T_n\}$ be a sequence of self mappings on $X$. Assume that there exist two sequences $(a_n)$ and $(b_n)$ of elements of $X$ such that

\begin{align}\label{condchat1}
F(D(T_i(x),T_j(y))) \leq \  & F(\delta_{i,j}[D(x,T_j(y))+ D(y,T_i(x))])\\
&- F(\gamma_{i,j}\psi[D(x,T_j(y)),D(y,T_i(x))]) \nonumber
\end{align}

for $x,y\in X$ with $x\neq y,$ $0\leq \delta_{i,j},\gamma_{i,j}<1 , \ i,j = 1,2,\cdots ,$ and for some $F \in \Phi$ homogeneous with degree $s$, where $\delta_{i,j}=D(a_i,a_j)$, $\gamma_{i,j}=D(b_i,b_j)$, and $\psi:[0,\infty)^2 \to [0,\infty)$ is a continuous mapping such that $%
\psi(x,y)=0$ if and only if $x=y=0$. If the sequence $(s_n)$ where $s_i=\frac{\delta_{i,i+1}^s }{1-\delta_{i,i+1}^s}$ is a non-increasing $\lambda$-sequence of elements of $\mathbb{R}^+=[0,+\infty)$ endowed with the $\max$ metric, then $\{T_n\}$ has a unique common fixed point in $X$. 

\end{theorem}

\begin{proof}
	The proof follows exactly the same steps as the proof of Theorem \ref{main}
\end{proof}

\begin{theorem}\label{chat2}
	Let $(X,D,K)$ be a complete metric type space and $\{T_n\}$ be a sequence of self mappings on $X$. Assume that there exist two sequences $(a_n)$ and $(b_n)$ of elements of $X$ such 
	
	\begin{align}\label{condchat2}
	F(D(T_i(x),T_j(y))) \leq \  & F(\delta_{i,j}[D(x,T_j(x))+ D(y,T_i(y))+D(x,y)])\\
	&- F(\gamma_{i,j}\psi[D(x,T_j(y)),D(y,T_i(x)),D(x,y)]) \nonumber
	\end{align}
	
	for $x,y\in X$ with $x\neq y,$ $0\leq \delta_{i,j},\gamma_{i,j}<1 , \ i,j = 1,2,\cdots ,$ and for some $F \in \Phi$ homogeneous with degree $s$, where $\delta_{i,j}=D(a_i,a_j)$, $\gamma_{i,j}=D(b_i,b_j)$, and $\psi:[0,\infty)^3 \to [0,\infty)$ is a continuous mapping such that $%
	\psi(x,y,z)=0$ if and only if $x=y=z=0$. If the sequence $(s_n)$ where $s_i=\frac{2^s\delta_{i,i+1}^s }{1-\delta_{i,i+1}^s}$ is a non-increasing $\lambda$-sequence of elements of $\mathbb{R}^+=[0,+\infty)$ endowed with the $\max$ metric,
	then $\{T_n\}$ has a unique common fixed point in $X$. 
	
\end{theorem}

\begin{proof}
	The proof follows exactly the same steps as the proof of Theorem \ref{main2}
\end{proof}

Like in the case of Theorem \ref{main}, the condition \eqref{chat1} can also be relaxed and we obtain:

\begin{theorem}\label{chat1relaxed}
	Let $(X,D,K)$ be a complete metric type space and $\{T_n\}$ be a sequence of self mappings on $X$ such that
	
	\begin{align}\label{condchat1relaxed}
	F(D(T_i(x),T_j(y))) \leq \  & F(\delta_{i,j}[D(x,T_j(y))+ D(y,T_i(x))])\\
	&- F(\gamma_{i,j}\psi[D(x,T_j(y)),D(y,T_i(x))]) \nonumber
	\end{align}
	
	for $x,y\in X$ with $x\neq y,$ $0\leq \delta_{i,j}, \gamma_{i,j}, \ i,j = 1,2,\cdots ,$ for some $F \in \Phi$ homogeneous with degree $s$ and $\psi:[0,\infty)^2 \to [0,\infty)$ is a continuous mapping such that $%
	\psi(x,y)=0$ if and only if $x=y=0$.
	If 
	\begin{itemize}
		\item[i)] for each $j,$ $\underset{i \to \infty}{\limsup}\ \delta_{i,j}^s <1, $
		
		\item[ii)] $$\sum_{n=1}^{\infty} C_n < \infty \text{ where } C_n = \prod\limits_{i=1}^n \frac{\delta_{i,i+1}^s}{1-\delta_{i,i+1}^s}, $$
	\end{itemize}
	
	then $\{T_n\}$ has a unique common fixed point in $X$. 
	
\end{theorem}		

Also, it is not necessary to point out that this relaxed condition can be applied to modify the hypotheses of Theorem \ref{chat2}.

\begin{corollary}
	In addition to hypotheses of Theorem \ref{chat1relaxed},
	suppose that for every $n \geq 1$, there exists $k(n) \geq 1 $ such
	that $a_{n,k(n)}< \frac{1}{2}$; then every $T_n$ has a unique fixed point in $X$.
\end{corollary}

\begin{remark}
	This corollary emphasizes on the subtile fact that in addition to the existence of a unique common fixed for the family $\{T_n\}$, each single map $T_n$ has a unique fixed point, which happens to be the common one. Also a similar corollary can be formulated\footnote{We shall suppose that for every $n \geq 1$, there exists a $k(n) \geq 1 $ such
		that $a_{n,k(n)}< \frac{1}{3}$.} in the relaxed form of Theorem \ref{chat2}.
\end{remark}

\begin{proof}
From Theorem \ref{chat1relaxed}, we know that the family $\{T_n\}$ has a
unique common fixed point $x^* \in X$. If $y^*$ is a fixed
point of a given $T_m$ then
\begin{align*}
	D(x^*,y^*) = D(T_{k(m)x^*,}T_my^*) & \leq a_{k(m),m}[D(x^*,T_my^*)+ D(y^*,T_{k(m)}x^*)] \\
	& = a_{k(m),m}[D(x^*,y^*)+ D(y^*,x^*)] \\
	& < D(x^*,y^*),
	\end{align*}

which implies $D(x^*,y^*) =0 $; which gives
the desired result.

\end{proof}

\begin{example}
We endow the set $X=[0,1]$ with the usual distance $d(x,y)= |x-y|$ for $x,y,\in X$. Then $(X,d)$ is a complete metric space. We set $F(x)=x$ whenever $x\in X$, then $F \in \Phi$ homogeneous with degree 1. For $n\geq$ we define the family $\{T_n\}$ of self-mappings on $X$ by

\[T_n= \left\{
\begin{array}{ll}
1 & : 0 < x \leq 1,\\
\frac{2}{3} + \frac{1}{n+2} & : x = 0.
\end{array}
\right.
\]

Using the notation of Theorem \ref{chat1relaxed}, we use the family of reals 
$$\delta_{i,j} = \frac{1}{3} + \frac{1}{|i-j|+6};\qquad \gamma_{i,j}=0.$$

Then clearly for each $j,$ 
$$\underset{i \to \infty}{\limsup}\ \delta_{i,j}^s <1, $$ 

and

$$C_n = \prod \limits_{i=1}^n \frac{\delta_{i,i+1}^s}{1-\delta_{i,i+1}^s} =  \left(\frac{10}{11}\right)^n ,$$
which is a convergent series.

Also it is a straightforward calculation to verify that 

\begin{align}
F(D(T_i(x),T_j(y))) \leq \  & F(\delta_{i,j}[D(x,T_j(y))+ D(y,T_i(x))])\\
&- F(\gamma_{i,j}\psi[D(x,T_j(y)),D(y,T_i(x))]) \nonumber
\end{align}

i.e. equivalently 

\begin{align}
D(T_i(x),T_j(y))\leq \  & \delta_{i,j}[D(x,T_j(y))+ D(y,T_i(x))].
\end{align}

One just has to consider the three possible cases: 

\begin{itemize}
	\item $x\in (0,1]$ and $y\in (0,1]$;
	\item $x\in (0,1]$ and $y=0$;
	\item $x=y=0$ with $i\neq j$.
\end{itemize}

So all the conditions of Theorem \ref{chat1relaxed} are satisfied and note that $x = 1$ is the only fixed point for the $T_n$'s.
\end{example}

	\section{Conclusion and future work}
Using the same idea as in the proofs of Theorem \ref{main} and Theorem \ref{main2}, one can establish the
following results.

\begin{theorem}

Let $(X,D,K)$ be a complete metric type space and $\{T_n\}$ be a sequence of self mappings on $X$. Assume that there exist two sequences $(a_n)$ and $(b_n)$ of elements of $X$ such 

\begin{align}
F(D(T_i^p(x),T_j^p(y))) \leq \  & F(\delta_{i,j}[D(x,T_i^p(x))+ D(y,T_j^p(y))])\\
&- F(\gamma_{i,j}\psi[D(x,T_i^p(x)),D(y,T_j^p(y))]) \nonumber
\end{align}

for $x,y\in X$ with $x\neq y,$ $p\geq 1,$ $0\leq \delta_{i,j},\gamma_{i,j}<1 , \ i,j = 1,2,\cdots ,$ and for some $F \in \Phi$ homogeneous with degree $s$, where $\delta_{i,j}=D(a_i,a_j)$, $\gamma_{i,j}=D(b_i,b_j)$, and $\psi:[0,\infty)^2 \to [0,\infty)$ is a continuous mapping such that $%
\psi(x,y)=0$ if and only if $x=y=0$. If the sequence $(s_n)$ where $s_i=\frac{\delta_{i,i+1}^s }{1-\delta_{i,i+1}^s}$ is a non-increasing $\lambda$-sequence of elements of $\mathbb{R}^+=[0,+\infty)$ endowed with the $\max$ metric, 
then $\{T_n\}$ has a unique common fixed point in $X$. 

\end{theorem}

	\begin{theorem}
	Let $(X,D,K)$ be a complete metric type space and $\{T_n\}$ be a sequence of self mappings on $X$. Assume that there exist two sequences $(a_n)$ and $(b_n)$ of elements of $X$ such 
	
	\begin{align}
	F(D(T_i^p(x),T_j^p(y))) \leq \  & F(\delta_{i,j}[D(x,T_i^p(x))+ D(y,T_j^p(y))+D(x,y)])\\
	&- F(\gamma_{i,j}\psi[D(x,T_i^p(x)),D(y,T_j^p(y)),D(x,y)]) \nonumber
	\end{align}
	
	for $x,y\in X$ with $x\neq y,$ $p\geq 1,$ $0\leq \delta_{i,j},\gamma_{i,j}<1 , \ i,j = 1,2,\cdots ,$ and for some $F \in \Phi$ homogeneous with degree $s$, where $\delta_{i,j}=D(a_i,a_j)$, $\gamma_{i,j}=D(b_i,b_j)$, and $\psi:[0,\infty)^3 \to [0,\infty)$ is a continuous mapping such that $%
	\psi(x,y,z)=0$ if and only if $x=y=z=0$. If the sequence $(s_n)$ where $s_i=\frac{2^s\delta_{i,i+1}^s }{1-\delta_{i,i+1}^s}$ is a non-increasing $\lambda$-sequence of elements of $\mathbb{R}^+=[0,+\infty)$ endowed with the $\max$ metric,
	then $\{T_n\}$ has a unique common fixed point in $X$. 
	
\end{theorem}	
	
Moreover, the above two generalisations also apply to Theorem \ref{chat1} and Theorem \ref{chat2} respectively.	
	
\vspace*{0.5cm}	
	
Recently the so-called $C$-class functions, which was introduced by A.H.Ansari \cite{aha} in 2014 and covers a large class of contractive conditions, has been applied successfully in the generalization of certain contractive conditions.

We read
\begin{definition}
	\label{C-class}(\cite{aha}) A continuous function $f:[0,\infty
	)^{2}\rightarrow \mathbb{R}$ is called \textit{$C$-class} function if \ for
	any $s,t\in \lbrack 0,\infty ),$ the following conditions hold:
	
	(1) $f(s,t)\leq s$;
	
	(2) $f(s,t)=s$ implies that either $s=0$ or $t=0$.
	
	We shall denote by $\mathcal{C}$ the collection of $C$-class functions.	
\end{definition}

\begin{example}
	\label{C-class examp}(\cite{aha}) The following examples show that the class $%
	\mathcal{C}$ is nonempty:

	\begin{enumerate}
		\item $f(s,t)=s-t.$
		
		\item $f(s,t)=ms,$ for some $m\in (0,1).$
		
		\item $f(s,t)=\frac{s}{(1+t)^{r}}$ for some $r\in (0,\infty ).$
		
		\item $f(s,t)=\log (t+a^{s})/(1+t)$, for some $a>1.$
		
	\end{enumerate}

\end{example}	
	
Therefore, the authors plan to study, in another manuscript, the existence of common fixed points for a family of mappings $T_i: (X,D,K)\to (X,D,K), i=1,2,\cdots$, defined on a metric type space $(X,D,K)$, which satisfy :

	\begin{align}\label{aha}
	F(D(T_i(x),T_j(y))) \leq \  & f(F(\delta_{i,j}[D(x,T_i(x))+ D(y,T_j(y))+D(x,y)]),\\
	& F(\gamma_{i,j}\psi[D(x,T_i(x)),D(y,T_j(y)),D(x,y)])) \nonumber
	\end{align}
for $x,y\in X$ with $x\neq y,$ where 

\begin{enumerate}
	
	\item $f \in \mathcal{C}$,
	\item  $(a_n)$ and $(b_n)$ are sequences of elements of $X$ and $\delta_{i,j}=D(a_i,a_j)$, $\gamma_{i,j}=D(b_i,b_j)$ with $0\leq \delta_{i,j},\gamma_{i,j}<1 , \ i,j = 1,2,\cdots ,$
	\item $F \in \Phi$ homogeneous with degree $s$
	
	\item $\psi:[0,\infty)^3 \to [0,\infty)$ is a continuous mapping such that $%
	\psi(x,y,z)=0$ if and only if $x=y=z=0$.
\end{enumerate}

under the condition that $s_i=\frac{2^s\delta_{i,i+1}^s }{1-\delta_{i,i+1}^s}$ is a non-increasing $\lambda$-sequence of elements of $\mathbb{R}^+=[0,+\infty)$ endowed with the $\max$ metric.

Moreover, if the existence of a common fixed point is established in that setting, therefore Theorems \ref{main} and \ref{main2} become direct corollaries by just setting $f(s,t)=s-t$ and the same applies to their equivalent asymmetric formulations.

	\section*{Conflict of interest.}
	
	The authors declare that there is no conflict
	of interests regarding the publication of this article.


	\bibliographystyle{amsplain}

\end{document}